\documentclass{artjlt}

\usepackage[utf8]{inputenc}
\usepackage{graphicx}
\usepackage{ amssymb ,  amsthm, amsmath}
\usepackage[colorlinks]{hyperref}
\usepackage{mathtools}
\usepackage{pst-node}
\usepackage{tikz-cd}
\usepackage[T1]{fontenc}
\usepackage[utf8]{inputenc}

\newtheorem{innercustomthm}{Theorem}
\newenvironment{customthm}[1]
  {\renewcommand\theinnercustomthm{#1}\innercustomthm}
  {\endinnercustomthm}
\numberwithin{equation}{section}

\title{On Rigidity of {$S$}-Arithmetic Kac-Moody Groups}                                     
\author{Amir Farahmand Parsa and Ralf K\"ohl}                 
\lastname{Farahmand Parsa and K\"ohl}  



\address{%
Amir Farahmand Parsa\\               
School of Mathematics, Institute for Research in Fundamental Sciences (IPM), P.O. Box: 19395-5746, Tehran, Iran, School of Mathematics, TATA Institute of Fundamental Research (TIFR), 400 005, Mumbai, India\\            
a.parsa@ipm.ir             
}

\address{%
Ralf K\"ohl\\               
Mathematisches Institut,
Justus Liebig Universit\"at Giessen
Arndtstrasse 2,
35392 Giessen, Germany\\            
ralf.koehl@math.uni-giessen.de                
}
\begin{document}


\maketitle

\begin{abstract}
In this article we investigate rigidity properties of $S$-arithmetic Kac-Moody groups in characteristic $0$.
\end{abstract}

\section{Introduction}
In this article we continue and generalize the study of rigidity properties of arithmetic Kac--Moody groups initiated in \cite{parsa2015strong} and \cite{zbMATH06667742}.
Notably, we confirm that the following super-rigidity and strong rigidity results hold also in the context of $S$-arithmetic Kac--Moody groups in characteristic $0$. For basic definitions we refer to Section~\ref{secKM}.
\begin{customthm}{\ref{thmsrbc}}[Super-rigidity]
Let $\mathcal{G}$ be a simply connected irreducible $2$-spherical split Kac-Moody functor of rank $\geq 2$.\ Let $\mathcal{H}$ be a centred $2$-spherical split Kac-Moody functor. Let $\Phi^{\mathcal{G}_{re}}$ be the set of real roots of $\mathcal{G}$. Let $$\phi:\Gamma(\mathcal{O}_{S})\to\mathcal{H}(l)$$ be a group homomorphism where $\Gamma(\mathcal{O}_S)$ is defined as (\ref{eqGammaOS}) and $l$ is a local field of characteristic  0. Then either
\begin{itemize}
    \item [(a)] there exists an ideal $0\not= q\trianglelefteq\mathcal{O}_{S}$ such that $ \Gamma_{\Pi}(q)$ (see (\ref{eqgammapiq})) is contained in the kernel of $\phi$ for a system of simple roots $\Pi\subset\Phi^{\mathcal{G}_{re}},$ or,
    
    \item[(b)] there exists a system of simple roots $\Pi\subset\Phi^{\mathcal{G}_{re}}$ such that either
\begin{itemize}
    \item[(i)]for some $\nu\in S\backslash\infty$, $l$ is a Galois extension of $K_{\nu}$ and there exists a continuous (with respect to both the Kac-Peterson and the weak Zariski topologies) group homomorphism $$\phi_{K_{\nu}}:\mathcal{G}(K_{\nu})\to\mathcal{H}(K_{\nu}),$$ which is {uniquely determined} by the property that there exists an ideal $0\not= q\trianglelefteq\mathcal{O}_{S}$ such that
    \begin{equation}
  \tag{VEP}
    \phi_{l}|_{\Gamma_{\Pi}(q)}=\phi|_{\Gamma_{\Pi}(q)};\label{introvirextprop}
    \end{equation}
    \item[(ii)] or, $l=\mathbb{R}, \mathbb{C}$ and there exists a continuous (with respect to both the Kac-Peterson and the weak Zariski topologies) group homomorphism $$\phi_{l}:\mathcal{G}(l)\to\mathcal{H}(l),$$ which is {uniquely determined} by the property that there exists an ideal $0\not= q\trianglelefteq\mathcal{O}_{S}$ such that
    \begin{equation}
  \tag{VEP}
    \phi_{l}|_{\Gamma_{\Pi}(q)}=\phi|_{\Gamma_{\Pi}(q)}.\label{introvirextprop0}
    \end{equation}
    \end{itemize}
    \end{itemize}
\end{customthm}
The dichotomy between (a) and (b) in the above theorem resembles the classical dichotomy that a homomorphic image of an $S$-arithmetic group is either finite (and hence by the congruence subgroup property has a congruence subgroup in its kernel) or allows for a virtual homomorphic extension to the ambient locally compact group.

The general strategy is to localize arguments to fundamental rank $2$ subgroups of split Kac-Moody groups which are in turn Chevalley groups and then use classical rigidity results for such Chevalley groups to obtain local virtual extension maps and then glue these local extensions together in order to obtain a global ``virtual'' extension. To localize the arguments we prove a fixed point lemma (Lemma~\ref{lemfptbs}) using Tavgen's bounded generation result (see \cite{zbMATH04142315}) and the Davis realization of the twin buildings associated to Kac-Moody groups over local fields (see \cite{zbMATH05218470}). Making use of Tavgen's bounded generation result is the main place where we substantially use the hypothesis of characteristic $0$. Then we use Margulis super-rigidity (see \cite[Theorem VIII 3.12, Theorem VIII(B)]{zbMATH00049189}) to build local virtual extensions. Here we need the hypotheses of rank $2$ and of $2$-sphericity in order to guarantee the existence of suitable rank $2$ subgroups of the involved Kac-Moody groups to which Margulis super-rigidity applies in our local-to-global approach. Finally, to obtain a global ``virtual'' continuous extension, we use a topological Curtis-Tits Theorem (see \cite[Theorem 2.20]{zbMATH06667742}) where $2$-sphericity again is essential.

By applying this super-rigidity twice we obtain the following strong rigidity.
\begin{customthm}{\ref{thmstrigbc}}[Strong rigidity]
Let $\mathcal{G}$ and $\mathcal{G}'$ be simply connected irreducible $2$-spherical split Kac–Moody functors of rank $\geq 2$. Let $K$ be an algebraic number field and let $S$ and $S'$ be two finite sets of places containing all the archimedean places and at least one non-archimedean for each of them. Let either $\phi:\Gamma(\mathcal{O}_S)\to\Gamma(\mathcal{O}_{S'})$ or $\phi:\mathcal{G}(\mathcal{O}_S)\to\mathcal{G}'(\mathcal{O}_{S'})$ be abstract group isomorphisms. Then $S=S'$ and for each place $\nu\in S$ there exists a continuous (with respect to both the Kac-Peterson and the weak Zariski topologies) group isomorphism $$\tilde{\phi}:\mathcal{G}(K_{\nu})\to\mathcal{G}'(K_{\nu}),$$
such that it satisfies (\ref{introvirextprop}). Moreover, $K$-points are preserved by this extension namely, $$\tilde{\phi}|_{\mathcal{G}(K)}:\mathcal{G}(K)\to\mathcal{G}'(K),$$
is an isomorphism. Hence the Kac-Moody functors are equal, i.e., $\mathcal{G}=\mathcal{G}'$.
\end{customthm}
\noindent
\textbf{Acknowledgements.} The research for this article was finished during the second author's visit to TIFR, Mumbai in April 2019 where the first author held a postdoctoral fellowship at that time -- the second author expresses his sincere gratitude for the hospitality during his visit. The second author also acknowledges partial financial support by Deutsche Forschungsgemeinschaft through the project KO4323/13. The first-named author would like to thank M.S. Raghunathan for his comments and consultations regarding this article. Both authors would like to thank T.N. Venkataramana for many helpful conversations and Timoth\'ee Marquis and Stefan Witzel for their comments on the draft. They would like to also thank the referee for their valuable and constructive remarks and comments on the manuscript.

\section{Kac-Moody functors}\label{secKM}
This section is to define split Kac-Moody functors which are the main ingredient of this article. The rest of the terms that are not explicitly defined here can be found in  \cite{zbMATH06667742} and the references therein.

Let $A$ be a generalized Cartan matrix of size $n$, i.e., an integral matrix $A=(a_{ij})_{1\leq i,j\leq n}$ such that $a_{ii}=2$ for all $1\leq i\leq n$, $a_{ij}\leq 0$  for $i\not=j$ and $a_{ij}=0$ if and only if $a_{ji}=0$ for all $1\leq i,j\leq n.$ A quintuple $\mathcal{D}=(I,A,\Lambda,\{c_{i}\}_{i\in I},\{h_{i}\}_{i\in I})$ is called a Kac-Moody root datum where $I=\{1,\cdots,n\}$, $A$ is a generalized Cartan matrix, $\Lambda$ is a free $\mathbb{Z}$-module, $\{c_{i}\}_{i\in I}\subset\Lambda$, and $\{h_{i}\}_{i\in I}$ is a subset of the $\mathbb{Z}$-dual  $\Lambda^{\vee}$ of $\Lambda$ such that for all $i,j\in I$, $h_{i}(c_j)=a_{ij}.$

In \cite{zbMATH04017219}, Tits associated a triplet $(\mathcal{G},\{\phi_{i}\}_{i\in I},\eta)$ to any Kac-Moody root datum $\mathcal{D}$ as follows.\ Here $\mathcal{G}$ is a group functor on the category of commutative unital rings, $\phi_{i}$ are maps $\text{\textbf{SL}}_{2}(-)\to\mathcal{G}(-),$ and $\eta$ is a natural transformation  $\text{Hom}_{\mathbb{Z}-\text{alg}}(\mathbb{Z}[\Lambda],-)\to\mathcal{G}$ such that the following assertions hold:
\begin{itemize}
    \item[(KMG1)] If $k$ is a field then $\mathcal{G}(k)$ is generated by the images of $\phi_{i}$ and $\eta(k).$
      \item[(KMG2)] For all rings $R$ the homomorphism $\eta(R):\text{Hom}_{\mathbb{Z}-\text{alg}}(\mathbb{Z}[\Lambda],R)\to\mathcal{G}(R)$ is injective.
        \item[(KMG3)] Given a ring $R$ and $u\in R^{\times}$, for every $i\in I$ one has $$\phi_{i}\begin{pmatrix} 
u & 0 \\
0 & u^{-1} 
\end{pmatrix}=\eta(\lambda\mapsto u^{h_i(\lambda)}),$$
where $\lambda\in\Lambda.$
          \item[(KMG4)] If $R$ is a ring, $k$ is a field and $\imath:R\to k$ is a monomorphism, then $\mathcal{\imath}:\mathcal{G}(R)\to\mathcal{G}(k)$ is a monomorphism.
            \item[(KMG5)] If $\mathfrak{g}$ is the complex Kac-Moody algebra of type $A,$ then there exists a homomorphism $\text{\textbf{Ad}}:\mathcal{G}(\mathbb{C})\to \text{Aut}(\mathfrak{g})$ such that $$\text{kernel}(\text{\textbf{Ad}})\subseteq\eta(\mathbb{C})(\text{Hom}_{\mathbb{Z}-\text{alg}}(\mathbb{Z}[\Lambda],\mathbb{C})),$$
            and for a given $z\in\mathbb{C}$ one has $$\text{\textbf{Ad}}\bigg(\phi_{i}\begin{pmatrix} 
1 & z \\
0 & 1 
\end{pmatrix}\bigg)=\text{exp}\big(\text{\textbf{Ad}}_{ze_i}\big),$$ $$\text{\textbf{Ad}}\bigg(\phi_{i}\begin{pmatrix} 
1 & 0 \\
z & 1 
\end{pmatrix}\bigg)=\text{exp}\big(\text{\textbf{Ad}}_{zf_i}\big);$$
where $\{e_i, f_i\}$ are part of a standard $\mathfrak{sl}_2$-triple for the fundamental Kac–Moody sub-Lie algebra corresponding to the simple root $\alpha_{i};$ furthermore, for every homomorphism $\gamma\in\text{Hom}_{\mathbb{Z}-\text{alg}}(\mathbb{Z}[\Lambda],\mathbb{C})$ one has $$\text{\textbf{Ad}}(\eta(\mathbb{C})(\gamma))(e_i)=\gamma(c_i)\cdot e_i,~~~~~\text{\textbf{Ad}}(\eta(\mathbb{C})(\gamma))(f_i)=\gamma(-c_i)\cdot f_{i}.$$
\end{itemize}
The group functor $\mathcal{G}$ associated to $\mathcal{D}$ is called a split Kac-Moody functor. The group functor $\mathcal{G}$ is simply connected if $\{h_i\}$ form a basis for $\Lambda^{\vee}$ and is centered if for every field $k$, $\mathcal{G}(k)$ is generated by the images of the $\phi_{i}.$ Note that
simple connected implies centered (see \cite{zbMATH04017219}).

\section{Super-rigidity}\label{secsuprig}
Let $K$ be an algebraic number field. Let $S$ be a finite set of places containing all the archimedeans, whose set is denoted by $\infty$. Denote by $\mathcal{O}_{S}$ the set of $S$-integers and by $\mathcal{O}_K$ the set of algebraic integers in $K$.  Let $\textbf{G}$ be a Chevalley group scheme. Define, $G_{S}:=\prod_{\nu\in S}\textbf{G}_{K_{\nu}}$, where $K_{\nu}$ is the completion of $K$ with respect to $\nu\in S$. 

\begin{lemma}[Fixed point theorem]\label{lemfptbs}
Let $\textbf{G}$ be an irreducible universal Chevalley group scheme of rank $\geq 2$. Let $K$ and $S$ be as above and further let $S$ contain at least one non-archimedean.\ Any action of $\textbf{G}(\mathcal{O}_{S})$ on a complete $CAT(0)$ polyhedral complex by cellular isometries has a global fixed point.
\end{lemma}
\begin{proof}
By \cite[Corollary 1]{zbMATH04142315}, $\textbf{G}(\mathcal{O}_{S})$ is boundedly generated by the finite set of its root subgroups, namely $$\mathcal{A}:=\{U_{\alpha}(\mathcal{O}_{S})~|~\alpha\in\Phi^{\textbf{G}}\},$$
where $\Phi^{\textbf{G}}$ is the set of roots of $\textbf{G}$. By \cite[Corollary 2.5]{zbMATH05541531} the lemma follows if each of the above root subgroups has a global fixed point. Let $\nu\in S$ be a non-archimedean place and let $p$ be the characteristic of the residue field of $K_{\nu}$. Then for each root $\alpha\in\Phi$, $U_{\alpha}(\mathcal{O}_{S})$ is a $p$-divisible abelian group and hence has a global fixed point by \cite[Corollary 2.7]{zbMATH05541531}.
\end{proof}

Recall a subgroup of a Kac-Moody group is called bounded if it is contained in the intersection of two spherical parabolic subgroups of opposite signs (see \cite[Corollary 12.67]{zbMATH05288866}).\ The following is derived from \cite{zbMATH03572914} and \cite[Theorem VIII 3.12]{zbMATH00049189}.
\begin{theorem}[Local super-rigidity]\label{thmlsrbc}
Let $\textbf{G}$ be an irreducible universal Chevalley group scheme of rank $\geq 2$.\ Let $K$ be any algebraic number field and $S$ contain at least one non-archimedean place. Let $\mathcal{H}$ be a centred $2$-spherical split Kac-Moody functor and let $l$ be a local field of characteristic  0. Assume that $$\phi:\textbf{G}(\mathcal{O}_{S})\to\mathcal{H}(l),$$ is an abstract group homomorphism. Then there exists a bounded subgroup $B$ of $\mathcal{H}(l)$ such that either:
\begin{itemize}
    \item[(i)] The image of $\text{Ad}|_{B}\circ\phi$ is relatively compact (for the adjoint map see \cite[\S 10.3]{zbMATH01762600}).
    \item[(ii)] For some $\nu\in S\backslash\infty$, $l$ is a Galois extension of $K_{\nu}$ and the map $\phi$ extends uniquely, continuously (with respect to both the Kac-Peterson and the (weak) Zariski topologies) and virtually to $$\tilde{\phi}:\textbf{G}_{K_{\nu}}\to B.$$
    \item[(iii)] $l=\mathbb{R},~\text{or}~\mathbb{C}$ and $\phi$ virtually extends to the following (uniquely determined) continuous map (with respect to both the Kac-Peterson and the (weak) Zariski topologies): $$\textbf{G}_{l}\to B.$$ 
\end{itemize}
\end{theorem}

\begin{proof}
By Lemma~\ref{lemfptbs}, the image of $\phi$ is contained in a bounded subgroup $B$ of $\mathcal{H}(l)$ (see \cite[\S 2.1.4]{zbMATH05541531}). 
Define the composite map $$\eta:\textbf{G}(\mathcal{O}_{S})\xrightarrow{\phi}B\xrightarrow{\text{Ad}|_{B}}\textbf{H}_{l},$$
where $\textbf{H}$ is a linear algebraic $l$-group. Without loss of generality we can assume that the image of $\eta$ is dense in $\textbf{H}$ which, by \cite[Theorem VIII(B)]{zbMATH00049189}, means $\textbf{H}$ is semisimple. Now by \cite[Theorem 3]{zbMATH03572914} and \cite[Theorem VIII 3.12, Theorem VIII(B)]{zbMATH00049189} (see also \cite[Theorem 10.1.5]{zbMATH03911340}) either:
\begin{itemize}
    \item[(a)] The image of $\eta$ is relatively compact in $\textbf{H}_{l}$.
        \item[(b)] For some $\nu\in S\backslash\infty$, $l$ is a Galois extension of $K_{\nu}$ and there exists a uniquely determined $K_{\nu}$-rational surjection $\bar{\eta}:\textbf{G}\to\textbf{H}$, such that the map $$\bar{\eta}|_{K_{\nu}}:\textbf{G}_{K_{\nu}}\to\textbf{H}_{K_{\nu}},$$
    is a virtual extension of $\eta$.
    \item[(c)] $l=\mathbb{R}~\text{or}~\mathbb{C}$ and there exists a uniquely determined $l$-rational surjection $\bar{\eta}:G_{\infty}\to\textbf{H}$ such that its restriction to $l$-rational points is a virtual extension of $\eta$.
\end{itemize}
Note that (a) yields (i). Also note that since the kernel of the adjoint map is contained in the center of $B$ which is $\text{Ad}|_{l}$-diagonalizable and lies inside a split torus of the split Kac-Moody group $\mathcal{H}(l)$ (see \cite[Proposition 3.6]{zbMATH05541531}), if the intersection of the image of $\phi$ with the kernel is infinite then by the normal subgroup theorem of $S$-arithmetic subgroups (see \cite[Theorem (IX)6.14]{zbMATH00049189}) the situation leads to the case (i) of the theorem. Therefore, when either (ii) or (iii) occurs the kernel of the adjoint map contains at most a finite subgroup of the image of $\phi$ (which is central, a priori).\ When the case (b) occurs, one can consider only the restriction of scalars to $K_{\nu}$ (i.e., considering $l=K_{\nu}$) using the functorial property of Chevalley group schemes and Kac-Moody functors over commutative unital rings. Also, since $\textbf{G}$ is absolutely almost simple and $\bar{\eta}$ is an $l$-rational surjection, we conclude $\bar{\eta}$ is actually a central isogeny and hence $K_{\nu}$-root subgroups of $\textbf{G}_{K_{\nu}}$ are mapped to $K_{\nu}$-root subgroups of $\textbf{H}_{K_{\nu}}$. Since $\bar{\eta}|_{K_{\nu}}$ is a virtual extension of $\eta$ by the injectivity of the adjoint map on root subgroups and the congruence subgroup theorem for $S$-arithmetic subgroups (see \cite[Section 6-7.5]{zbMATH01950200}) we conclude that for each $\alpha\in\Phi^{\textbf{G}}$ there exist a finite index subgroup $\mathcal{O}_S^{fin}$ of $\mathcal{O}_S$ (with respect to the additive structure, see e.g., \cite[Chapter 18]{zbMATH03305793}) and roots $\gamma\in\Phi^{\textbf{H}},~\delta\in\Phi^{B}$ such that $$\phi_{\alpha}:U_{\alpha}(\mathcal{O}_S^{fin})\xrightarrow{\eta}U_{\gamma}(K_{\nu})\xrightarrow{\text{Ad}^{-1}|_{U_{\delta}}}U_{\delta}(K_{\nu}).$$ 
Since $S$ contains at least one non-archimedean place, the ring $\mathcal{O}_S$, and hence $\mathcal{O}_S^{fin}$, is dense in $K_{\nu}$. Therefore, for each $\alpha\in\Phi^{\textbf{G}}$ there exists a uniquely determined continuous map $$\tilde{\phi}_{\alpha}:U_{\alpha}(K_{\nu})\xrightarrow{}U_{\delta}(K_{\nu}),$$
virtually extending $\phi_{\alpha}$. Note that $\tilde{\phi}_{\alpha}$ is defined without ambiguity since $\phi_\alpha$ is defined via the composition of $\eta$, which is continuous by super-rigidity, and  the map $\text{Ad}^{-1}|_{U_{\delta}}$ which is continuous since $\text{Ad}|_{U_{\delta}}$ is a continuous group isomorphism of $K_{\nu},$ hence open (see \cite[Corollary 7.16 (iii)]{zbMATH06308047} and Lemma~\ref{aplemtiso}).\ Now since $\textbf{G}$ is simply connected and the kernel of $\text{Ad}|_{B}$ is central we conclude (ii). Moreover, since the lifting maps $\tilde{\phi}_{\alpha}$ are polynomials by the proof of Lemma~\ref{aplemtiso}, the virtual extension is continuous with respect to the (weak) Zariski topologies.

In the case of (c) when $l=\mathbb{R},\mathbb{C}$ by looking at the component of $G_{\infty}$ corresponding to $l$ and co-restricting $\bar{\eta}$ to this component, the same argument as above produces (iii).
\end{proof}
\begin{remark}\label{remboundsemi}\leavevmode
\begin{itemize}
  
    \item[\textbf{1.}] Assume that case (i) of the theorem does not occur. Since $\textbf{G}$ is semisimple, $\textbf{G}_{K_{\nu}}$ is perfect, and so is the image of $\tilde{\phi}$.\ Therefore, without loss of generality, one can assume that $B$ lies inside the derived subgroup of a spherical Levi subgroup by \cite[Lemma 3.3 and Proposition 3.6]{zbMATH05072607} and preservation of the Levi decomposition of bounded subgroups in the Kac-Moody context and linear algebraic groups theory via the adjoint map, see \cite[\S 10.3]{zbMATH01762600}.
    \item[\textbf{2.}] {{Assume that case (ii) of the theorem occurs. If the image of $\phi$ happens to lie inside $\mathcal{H}(\mathcal{O}_{S'})$ for some set of $S'$-integers of $K$, then by Theorem~\ref{thmlsrbc} there exists a non-archimedean valuation $\nu\in S\cap S'$. In this setting, by the construction of the extension on the root subgroups, the image of the $K$-points lies inside $\mathcal{H}(K)$ by the extension $\tilde{\phi}$ (see Lemma~\ref{aplemtiso} and its proof).} In particular, in this setting if $B_{K}$ denotes the $K$-points of $B$ obtained via the natural embedding $\mathcal{H}(K)\xhookrightarrow{}\mathcal{H}(l)$, then $$\tilde{\phi}(\textbf{G}(K))\subset B_{K}.$$ Furthermore, by abuse of notation let $B_{K}:=\tilde{\phi}(\textbf{G}(K))$, since $\textbf{G}(K)$ is a Chevalley group there exist a $g\in\mathcal{H}(K)$ and a spherical sub-diagram $\Pi_{\text{sph}}\subset\Pi$ such that $B_{K}^{g}$ is a central quotient of $$\mathcal{H}_{\Pi_{\text{sph}}}(K):=\langle\mathcal{H}_{\alpha}(K)~|~\alpha\in\Pi_{\text{sph}}\rangle.$$} 
\end{itemize}
\end{remark}
\begin{Definition}
Let $\mathcal{G}$ be a simply connected irreducible $2$-spherical split Kac-Moody functor of rank $\geq 2$.\ Let $\mathcal{H}$ be a centred $2$-spherical split Kac-Moody functor. Let $\Phi^{\mathcal{G}_{re}}$ be a set of real roots of $\mathcal{G}$ and $\Pi$ be a system of its simple roots. Define the  \textbf{$S$-arithmetic Kac-Moody group}
\begin{equation}\label{eqGammaOS}
    \Gamma(\mathcal{O}_{S}):=\langle U_{\alpha}(\mathcal{O}_{S})~|~\alpha\in \Phi^{\mathcal{G}_{re}}\rangle.
\end{equation}
Let $\mathcal{G}_{\alpha,\beta}$ denote the universal Chevalley group functor obtained from $\mathcal{G}$ with respect to a pair of non-orthogonal simple roots $\alpha,\beta\in\Pi$. For an ideal $0\not= q\trianglelefteq\mathcal{O}_{S}$ let $\Gamma_{\alpha,\beta}(q)$ denote the kernel of the natural group homomorphism $$\mathcal{G}_{\alpha,\beta}(\mathcal{O}_{S})\to\mathcal{G}_{\alpha,\beta}(\mathcal{O}_{S}/q).$$ Note that $q$ always has finite index in $\mathcal{O}_{S}$ and hence $\Gamma_{\alpha,\beta}(q)$ has a finite index in $\mathcal{G}_{\alpha,\beta}(\mathcal{O}_{S})$ and hence forms an $S$-arithmetic subgroup. Now define 
\begin{equation}\label{eqgammapiq}
    \Gamma_{\Pi}(q):=\langle \Gamma_{\alpha,\beta}(q)~|~\alpha,\beta\in\Pi~\text{and}~\text{non-orthogonal}~\rangle.
\end{equation}
Let $$\phi: \Gamma(\mathcal{O}_{S})\to\mathcal{H}(l)$$ be a group homomorphism where $l$ is a local field of characteristic  0. 
We call a pair of non-orthogonal simple roots $\alpha,\beta\in\Pi$ \textbf{non-admissible} with respect to $\phi$ if the restriction of $\phi$ to 
\begin{equation}
    \mathcal{G}_{\alpha,\beta}(\mathcal{O}_{s}):=\langle U_{\lambda}(\mathcal{O}_{S})~|~\lambda\in(\mathbb{Z}\alpha+\mathbb{Z}\beta)\cap\Phi^{\mathcal{G}_{re}}\rangle,
\end{equation}
 which we denote by $\phi^{\alpha,\beta}$ is of type Theorem~\ref{thmlsrbc}(i). The pair is called \textbf{admissible} otherwise.
\end{Definition}

\begin{corollary}\label{corstrring}
Let $\textbf{G}$ be an irreducible universal Chevalley group scheme of rank $\geq 2$. Let $S$ and $S'$ be two finite sets of places of $K$ containing all archimedean places and let $S$ contain at least one non-archimedean place. Let $\mathcal{H}$ be a centred $2$-spherical split Kac-Moody functor. Assume that $$\phi:\textbf{G}(\mathcal{O}_{S})\to\mathcal{H}(\mathcal{O}_{S'}),$$ is an abstract injective group homomorphism. Then any system of simple roots $\Pi^{\textbf{G}}$ of the set of roots $\Phi^{\textbf{G}}$ of $\textbf{G}$ does not admit any non-admissible pairs of simple roots with respect to $\phi$ for any embedding $\mathcal{H}(\mathcal{O}_{S'})\xhookrightarrow{}\mathcal{H}(K_{\nu})$, ($\nu'\in S'$). Moreover, $S'\subseteq S.$ 
\end{corollary}
\begin{proof}
Since the image of $\phi$ and also the image of its restriction to rank 2 subgroups associated to pairs of non-orthogonal roots (see \cite[Theorem VIII(A)]{zbMATH00049189}) are always infinite non-abelian, then applying Theorem~\ref{thmlsrbc} to the composite
$$\phi_{\nu}:\textbf{G}(\mathcal{O}_{S})\xrightarrow{\phi}\mathcal{H}(\mathcal{O}_{S'})\xhookrightarrow{}\mathcal{H}(K_{\nu'}),$$ for any $\nu'\in S'$ leads to case (ii) or case (iii) of Theorem~\ref{thmlsrbc} hence the corollary. Moreover, Theorem~\ref{thmlsrbc}((ii) or (iii)), in turn, implies that there exists a place $\nu\in S$ such that $K_{\nu}=K_{\nu'}$. Therefore, $S'\subseteq S.$ 
\end{proof}
\begin{theorem}[Super-rigidity]\label{thmsrbc}
Let $\mathcal{G}$ be a simply connected irreducible $2$-spherical split Kac-Moody functor of rank $\geq 2$. Let $\mathcal{H}$ be a centred $2$-spherical split Kac-Moody functor. Let $\Phi^{\mathcal{G}_{re}}$ be the set of real roots of $\mathcal{G}$. Let $$\phi:\Gamma(\mathcal{O}_{S})\to\mathcal{H}(l)$$ be a group homomorphism where $l$ is a local field of characteristic  0. Then either
\begin{itemize}
    \item [(a)] there exists an ideal $0\not= q\trianglelefteq\mathcal{O}_{S}$ such that $ \Gamma_{\Pi}(q)$ is contained in the kernel of $\phi$ for a system of simple roots $\Pi,$ or,
    
    \item[(b)] there exists a system of simple roots $\Pi$ such that either
\begin{itemize}
    \item[(i)]for some $\nu\in S\backslash\infty$, $l$ is a Galois extension of $K_{\nu}$ and there exists a continuous (with respect to both the Kac-Peterson and the weak Zariski topologies) group homomorphism $$\phi_{K_{\nu}}:\mathcal{G}(K_{\nu})\to\mathcal{H}(K_{\nu}),$$ which is {uniquely determined} by the property that there exists an ideal $0\not= q\trianglelefteq\mathcal{O}_{S}$ such that
    \begin{equation}
  \tag{VEP}
    \phi_{l}|_{\Gamma_{\Pi}(q)}=\phi|_{\Gamma_{\Pi}(q)};\label{virextprop}
    \end{equation}
    \item[(ii)] or, $l=\mathbb{R}, \mathbb{C}$ and there exists a continuous (with respect to both the Kac-Peterson and the weak Zariski topologies) group homomorphism $$\phi_{l}:\mathcal{G}(l)\to\mathcal{H}(l),$$ which is {uniquely determined} by the property that there exists an ideal $0\not= q\trianglelefteq\mathcal{O}_{S}$ such that
    \begin{equation}
  \tag{VEP}
    \phi_{l}|_{\Gamma_{\Pi}(q)}=\phi|_{\Gamma_{\Pi}(q)}.
    \end{equation}
    \end{itemize}
    \end{itemize}
\end{theorem}

\noindent Notation: We call the property (\ref{virextprop}) from the preceding theorem the \textbf{Virtual Extension Property}. Note however that principal congruence subgroups in arithmetic non-spherical Kac-Moody groups are not of finite index.

\begin{proof}
First assume that there exists a system of simple roots $\Pi$  such that it does not contain any non-admissible pairs with respect to $\phi$. Let $\alpha,\beta\in\Pi$ be a pair of non-orthogonal simple roots. For the restriction map $\phi^{\alpha,\beta}$ we are in the situation of either (ii) or (iii) of Theorem~\ref{thmlsrbc}. When case (ii) occurs by restriction of scalars as in the proof of Theorem~\ref{thmlsrbc}, without loss of generality we assume $l=K_{\nu}$.\ Now by \cite[Theorem 2.20]{zbMATH06667742}, it suffices to check that the local extensions coincide on the overlaps and the obtained global extension satisfies (\ref{virextprop}). To this end, let $\alpha_1,\alpha_2,\alpha_3$ be three simple roots in $\Pi$ such that $\alpha_1$ and $\alpha_3$ are non-orthogonal to $\alpha_2$. For such a triplet we have $$\mathcal{G}_{\alpha_1,\alpha_2}(l)\cap\mathcal{G}_{\alpha_2,\alpha_3}(l)=\mathcal{G}_{\alpha_2}(l)\cong\text{SL}_{2}(l),$$
since the Kac-Moody group functor $\mathcal{G}$ is simply connected and $2$-spherical. For $i=1,2$ let $\phi^{\alpha_i,\alpha_{i+1}}_{l}$ denote the extension of $\phi^{\alpha_i,\alpha_{i+1}}$ from Theorem~\ref{thmlsrbc}. First note that the lifting argument in Theorem~\ref{thmlsrbc} implies that if we denote by $B^{\alpha_i,\alpha_{i+1}}$ the bounded subgroup containing the image of $\phi^{\alpha_i,\alpha_{i+1}}$ then $$\phi^{\alpha_1,\alpha_{2}}_{l}(\mathcal{G}_{\alpha_2}(l))\subset B^{\alpha_2,\alpha_3}.$$ Moreover, by the virtual extension property and the congruence subgroup theorem for $S$-arithmetic subgroups (see \cite[Section 6-7.5]{zbMATH01950200}) there exists an ideal $0\not=q_{\alpha_i,\alpha_{i+1}}\trianglelefteq\mathcal{O}_{S}$ such that
\begin{equation*}\label{eqthmsrgoverlap}
  \phi^{\alpha_i,\alpha_{i+1}}_{l}|_{\Gamma_{\alpha_i,\alpha_{i+1}(q_{\alpha_i,\alpha_{i+1}})}}=\phi^{\alpha_i,\alpha_{i+1}}|_{\Gamma_{\alpha_i,\alpha_{i+1}(q_{\alpha_i,\alpha_{i+1}})}}
\end{equation*}
Moreover, the extensions are uniquely determined by this property. By the lifting argument the uniqueness still holds if we replace the ideal $q_{\alpha_i,\alpha_{i+1}}$ by another ideal that is contained in it. Let $m:=q_{\alpha_1,\alpha_{2}}\cdot q_{\alpha_2,\alpha_{3}}$. Both of the extensions coincide on $\Gamma_{\alpha_2}(m)$ so do on $\mathcal{G}_{\alpha_2}(l)$. This means we have a global extension $\phi_{l}:\mathcal{G}(l)\to\mathcal{H}(l)$. By the above argument for every non-orthogonal pair $\alpha,\beta$ of simple roots in $\Pi$ there exists a local extension $\phi^{\alpha,\beta}_{l}$ which is uniquely determined by the {virtual extension property}, namely (\ref{virextprop}) with respect to an ideal $q_{\alpha,\beta}\trianglelefteq\mathcal{O}_{S}$. Let $q$ be the product of all of such ideals. The above argument on the overlaps of local extensions implies that such local extensions are not only compatible on the overlaps but also they agree on $\Gamma_{\Pi}(q).$ Then by the construction of the global extension, (\ref{virextprop}) holds for the ideal $q$ and the global extension is uniquely determined by this property. Moreover, since in the non-archimedean case the local extensions on fundamental irreducible rank 2 subgroups are continuous with respect to the weak Zariski topology by Theorem~\ref{thmlsrbc}, and  the weak Zariski topology is defined by roots subgroups which are conjugates of fundamental irreducible rank 2 subgroups, the global extension is continuous also with respect to the weak Zariski topology.

To finish the proof, one first needs to show that if there exists a system of simple roots containing an admissible pair with respect to $\phi$, then \textbf{every} non-orthogonal pair of the system of simple roots is admissible.\ In particular, if there exists a non-admissible pair then all pairs are non-admissible. But if there exists an admissible pair of simple roots, namely $\{\alpha,\beta\}$, then applying the local super-rigidity, Theorem~\ref{thmlsrbc} to $\mathcal{G}_{\alpha,\beta}$, (\ref{virextprop}) implies that any non-orthogonal pair of simple roots containing either $\alpha$ or $\beta$ has infinite image and hence is admissible.\ Now the claim follows from the irreducibly of the underlying Kac-Moody group functor. Therefore, if there exists a non-admissible pair then all pairs are non-admissible and consequently the kernel of the restriction $\phi$ to any non-orthogonal pair of simple roots $\alpha,\beta$ contains $\mathcal{G}_{\alpha,\alpha}(q_{\alpha,\beta})$ by the congruence subgroup property for $S$-arithmetic Chevalley groups for some $0\not=q_{\alpha,\beta}\trianglelefteq\mathcal{O}_{S}$. Hence in this situation $\Gamma_{\Pi}(q)$ is contained in the kernel of $\phi$ for $q:=\prod q_{\alpha,\beta}$ where $\alpha,\beta$ run through all non-orthogonal pairs of simple roots in $\Pi.$ 
\end{proof}

\begin{remark}\label{remrapits}\leavevmode

\begin{itemize}
    \item[\textbf{1.}] Note that in the above theorem if there exists a system of simple roots containing at least one admissible pair with respect to $\phi$, then the extension exists and its existence is independent of the choice of a system of simple roots. But to show this, one first needs to establish a commensurability theorem (see Theorem~\ref{thmcombc} below). Therefore, this claim is proved later as Corollary~\ref{corsupringindsys}.
    
    \item[\textbf{2.}] Let $l$ be non-Archimedean. By the construction of local extensions in Theorem~\ref{thmlsrbc}, if the image of $\phi$ is contained in $\mathcal{H}(\mathcal{O}_{S'})$ for some set of $S'$-integers then the $K$-points of the local extensions (see Remark~\ref{remboundsemi}) and hence the global extension are mapped to $K$-points. In other words, $$\phi_{l}(\mathcal{G}(K))\subset\mathcal{H}(K).$$
\end{itemize}
\end{remark}

\section{Strong rigidity}\label{secstrig}
In this section $\mathcal{G}$ denotes an irreducible simply connected $2$-spherical split Kac-Moody group functor of rank $\geq 2$. In this setting $\Phi$ denotes the set of real roots associated to $\mathcal{G}$ and $\Pi$ is a set of simple roots in $\Phi$. Define$$\Gamma(q):=\langle U_{\alpha}(q)~|~\alpha\in\Phi\rangle,$$ and $$\Lambda_{\Pi}(q):=\langle U_{\alpha}(q)~|~\alpha\in\pm\Pi\rangle,$$
where $0\not= q\trianglelefteq\mathcal{O}_{S}$. Note that by definition $\Lambda_{\Pi}(q)$ is readily a subgroup of $\Gamma(q)$.
\begin{lemma}\label{lemspharith}
Let $\Pi$ be spherical of rank $\geq 2$. Then for any ideal $0\not= q\trianglelefteq\mathcal{O}_{S}$ the group $\Lambda_{\Pi}(q)$ has finite index in $\mathcal{G}(\mathcal{O}_{S})$ where $\mathcal{G}$ is simply connected.
\end{lemma}
\begin{proof}
First note that any non-trivial ideal has finite index in the additive group of $\mathcal{O}_{S}$ (see \cite[Chapter 18]{zbMATH03305793} or \cite[Chapter III]{zbMATH00756383}). Now, by the commutator relations of simply connected Chevalley groups (see \cite[p. 27]{zbMATH03111938} or \cite[Chapter 6]{zbMATH06680202}) for every root $\alpha\in\Phi$ there exists an ideal $0\not=q^{\alpha}$ contained in $q$ such that $U_{\alpha}(q^{\alpha})$ is contained in $\Lambda_{\Pi}(q)$. Let $q'$ be the multiple of all of such ideals $q^{\alpha}$ for the (finite) set of roots $\Phi$. This implies that $\Gamma(q')$ is contained in $\Lambda_{\Pi}(q)$. But by \cite[Theorem B]{zbMATH03541031}, $\Gamma(q')$ has finite index in $\mathcal{G}(\mathcal{O}_{S})$ and hence $\Lambda_{\Pi}(q)$ has finite index in $\Gamma(q)$. 
\end{proof}
Note here that the last sentence of the above proof shows why the terminology ``Virtual Extension Property'' has been chosen for the property (\ref{virextprop}) as in the spherical cases, the homomorphism $\phi_{l}$ is indeed a virtual extension.
\begin{lemma}\label{lemgamqinv}
For every $g\in\mathcal{G}(K)$ and $0\not= q\trianglelefteq\mathcal{O}_{S}$, there exists an ideal $0\not=\tilde{q}\trianglelefteq\mathcal{O}_{S}$ such that
$$\Lambda_{\Pi}(\tilde{q})\leq\Lambda_{\Pi}^{g}(q).$$
\end{lemma}
\begin{proof}
Since $\mathcal{G}$ is simply connected it suffices to show the lemma for an element $g\in U_{\alpha}(K)$ where $\alpha\in\pm\Pi$. Moreover, since $g$ acts trivially on any root subgroup $U_{\beta}(K)$ with $\beta$ orthogonal to $\pm\alpha$, it also suffices to investigate the action of $g$ on those root subgroups whose roots are in $\pm\Pi$ and non-orthogonal to $\alpha$. Let $\beta$ be non-orthogonal to $\alpha$. Note that $$g\in U_{\alpha}(K)\subset\text{Comm}_{\mathcal{G}_{\alpha,\beta}(K)}(\mathcal{G}_{\alpha,\beta}(\mathcal{O}_{S}))=\text{Comm}_{\mathcal{G}_{\alpha,\beta}(K)}(\Lambda_{\alpha,\beta}(q)),$$
where the last equality follows from Lemma~\ref{lemspharith}. Therefore, $\Lambda_{\alpha,\beta}^{g}(q)$ is an $S$-arithmetic subgroup. Hence by the congruence subgroup property (see e.g., \cite{zbMATH04196274}) there exists an ideal $0\not=q_{\alpha,\beta}\trianglelefteq\mathcal{O}_{S}$ such that $$\Lambda_{\alpha,\beta}(q_{\alpha,\beta})\leq\Gamma_{\alpha,\beta}(q_{\alpha,\beta})\leq\Lambda^{g}_{\alpha,\beta}(q)\leq\Lambda^{g}_{\Pi}(q).$$
Let $\tilde{q}:=\prod q_{\alpha,\beta}$ where $\beta$ runs through non-orthogonal roots to $\alpha$ in $\Pi$. It is clear that $$\Lambda_{\Pi}(\tilde{q})\leq\Lambda^{g}_{\Pi}(q).$$
\end{proof}
\begin{proposition}\label{propqinva}
For every $g\in\mathcal{G}(K)$ and $0\not= q\trianglelefteq\mathcal{O}_{S}$, there exist not-trivial ideals $q_1,q_2\trianglelefteq\mathcal{O}_{S}$ such that
$$\Gamma_{\Pi}(q_2)\leq\Lambda_{\Pi}(q_1)\leq\Lambda_{\Pi}^{g}(q)\leq\Gamma_{\Pi}^{g}(q).$$
\end{proposition}
\begin{proof}
By Lemma~\ref{lemgamqinv} it suffices to show for any ideal $0\not=q_1\trianglelefteq\mathcal{O}_{S}$ there exists an ideal $0\not=q_2\trianglelefteq\mathcal{O}_{S}$ such that  $$\Gamma_{\Pi}(q_2)\leq\Lambda_{\Pi}(q_1).$$ But for every pair of non-orthogonal simple roots $\alpha,\beta\in\Pi$ by the congruence subgroup property for $S$-arithmetic subgroups (see \cite[Section 6-7.5]{zbMATH01950200}) and Lemma~\ref{lemspharith} there exists an ideal $0\not=q_{\alpha,\beta}\trianglelefteq\mathcal{O}_{S}$ such that $$\Gamma_{\alpha,\beta}(q_{\alpha,\beta})\leq\Lambda_{\alpha,\beta}(q_1).$$ Hence the ideal $q_2:=\prod q_{\alpha,\beta}$ yields the desired result where $\alpha,\beta$ run through all pairs of non-orthogonal simple roots.
\end{proof}
\begin{theorem}[Local commensurability]\label{thmlcombc}
Let $\textbf{G}$ be an irreducible universal Chevalley group scheme of rank $\geq 2$.\ Assume $S$ contains at least one non-archimedean place. Let $\textbf{G}_{\mathcal{O}_{S}}$ be a subgroups of $\mathcal{G}(\mathcal{O}_{S})$ isomorphic to a central quotient of $\textbf{G}(\mathcal{O}_{S})$.\ Then for any ideal $0\not= q\trianglelefteq\mathcal{O}_{S}$, the group $\textbf{G}_{\mathcal{O}_{S}}\cap\Gamma_{\Pi}(q)$ has finite index in $\textbf{G}_{\mathcal{O}_{S}}$.  
\end{theorem}
\begin{proof} 
For a non-archimedean place $\nu\in S$, the assumptions provide the following maps 
$$\phi:\textbf{G}(\mathcal{O}_{S})\twoheadrightarrow\textbf{G}_{\mathcal{O}_{S}}\xhookrightarrow{}\mathcal{G}(\mathcal{O}_{S})\xhookrightarrow{}\mathcal{G}(K_{\nu}),$$
where $\phi$ is a surjective group homomorphism with central kernel. Then by Theorem~\ref{thmlsrbc} and Remark~\ref{remboundsemi} on obtains
$$\phi:\textbf{G}(\mathcal{O}_{S})\twoheadrightarrow\textbf{G}_{\mathcal{O}_{S}}\xhookrightarrow{}B_{K}\xhookrightarrow{}\mathcal{G}(K),$$ where $B_{K}$ is a bounded subgroup provided by Remark~\ref{remboundsemi}. Therefore, there exist a $g\in\mathcal{G}(K)$ and a spherical sub-diagram $\Pi_{\text{sph}}\subset\Pi$ such that $B_{K}^{g}$ is a central quotient of $$\mathcal{G}_{\Pi_{\text{sph}}}(K):=\langle\mathcal{G}_{\alpha}(K)~|~\alpha\in\Pi_{\text{sph}}\rangle.$$ Hence it suffices to show that $\mathcal{G}^{-g}_{\Pi_{\text{sph}}}(K)\cap\Gamma_{\Pi}(q)$ has finite index in $\mathcal{G}^{-g}_{\Pi_{\text{sph}}}(\mathcal{O}_{S}).$ To this end, first note that by Proposition~\ref{propqinva} there exists an ideal $q'\trianglelefteq\mathcal{O}_{S}$ such that $\Gamma_{\Pi}(q')\leq\Gamma^{g}_{\Pi}(q).$ Also, by Lemma~\ref{lemspharith} the intersection $\Gamma_{\Pi}(q')\cap\mathcal{G}_{\Pi_{\text{sph}}}(\mathcal{O}_{S})$ has finite index in $\mathcal{G}_{\Pi_{\text{sph}}}(\mathcal{O}_{S})$. But $$\Gamma_{\Pi}(q')\cap\mathcal{G}_{\Pi_{\text{sph}}}(\mathcal{O}_{S})\leq\Gamma^{g}_{\Pi}(q)\cap\mathcal{G}_{\Pi_{\text{sph}}}(\mathcal{O}_{S}),$$
which implies that $\Gamma^{g}_{\Pi}(q)\cap\mathcal{G}_{\Pi_{\text{sph}}}(\mathcal{O}_{S})$ has finite index in $\mathcal{G}_{\Pi_{\text{sph}}}(\mathcal{O}_{S})$ and consequently $\Gamma_{\Pi}(q)\cap\mathcal{G}^{-g}_{\Pi_{\text{sph}}}(\mathcal{O}_{S})$ has finite index in $\mathcal{G}^{-g}_{\Pi_{\text{sph}}}(\mathcal{O}_{S})$.
\end{proof}
\begin{theorem}[Strong rigidity]\label{thmstrigbc}
Let $\mathcal{G}$ and $\mathcal{G}'$ be simply connected irreducible $2$-spherical split Kac–Moody functors of rank $\geq 2$. Let $K$ be an algebraic number field and let $S$ and $S'$ be two finite sets of places containing all the archimedean places and at least one non-archimedean for each of them. Let either $\phi:\Gamma(\mathcal{O}_S)\to\Gamma(\mathcal{O}_{S'})$ or $\phi:\mathcal{G}(\mathcal{O}_S)\to\mathcal{G}'(\mathcal{O}_{S'})$ be abstract group isomorphisms. Then $S=S'$ and for each place $\nu\in S$ there exists a continuous (with respect to both the Kac-Peterson and the weak Zariski topologies) group isomorphism $$\tilde{\phi}:\mathcal{G}(K_{\nu})\to\mathcal{G}'(K_{\nu}),$$
such that it satisfies (\ref{virextprop}). Moreover, $K$-points are preserved by this extension namely, $$\tilde{\phi}|_{\mathcal{G}(K)}:\mathcal{G}(K)\to\mathcal{G}'(K),$$
is an isomorphism. Hence the Kac-Moody functors are equal, i.e., $\mathcal{G}=\mathcal{G}'$.
\end{theorem}
\begin{proof}
For any group functor $G$, let $\bar{G}(k)$ be the group modulo its centre via $\pi:G(k)\to\bar{G}(k):=G(k)/Z(G(k)).$ By the assumptions there exists either of the following isomorphisms (by abuse of notation): $$\phi:\bar{\Gamma}(\mathcal{O}_S)\to\bar{\Gamma'}(\mathcal{O}_{S'}),$$ or $$\phi:\bar{\mathcal{G}}(\mathcal{O}_S)\to\bar{\mathcal{G}'}(\mathcal{O}_{S'}).$$
By the isomorphism theorem over fields, \cite[Theorem A]{zbMATH05541531}, to prove the isomorphism theorem for $S$-arithmetic subgroups, it suffices to prove the theorem for $\bar{\mathcal{G}}$ and $\bar{\mathcal{G}}'.$ Let $\Pi^{\mathcal{G}'}$ denote a system of simple roots in the set of real roots $\Phi^{\mathcal{G}'}$ of $\mathcal{G}'$. By Corollary~\ref{corstrring} neither $\Pi^{\mathcal{G}}$ nor $\Pi^{\mathcal{G}'}$ contains any non-admissible pairs of simple roots with respect to $\phi$ and $\phi^{-1}$ respectively and $S=S'.$  Moreover, neither $\Pi^{\mathcal{G}}$ nor $\Pi^{\mathcal{G}'}$ contains any non-admissible pairs of simple roots with respect to $\phi\circ\pi$ and $\phi^{-1}\circ\pi$ respectively. This is because otherwise this would mean the image of irreducible rank 2 $S$-arithmetic subgroups are finite under $\phi\circ\pi$ and similarly $\phi^{-1}\circ\pi$ which would imply that the centre of such irreducible rank 2 subgroups are finite index subgroups which is a contradiction to the Normal subgroup theorem for $S$-arithmetic subgroups (see e.g., \cite[Theorem A and B]{zbMATH03541031} or \cite[Theorem (IX)6.14]{zbMATH00049189}). This allows one to apply Theorem~\ref{thmsrbc} to $\phi$ and $\phi^{-1}$ and also to the following maps
 $$\phi\circ\pi:{\Gamma}(\mathcal{O}_S)\to\bar{\Gamma'}(\mathcal{O}_{S'})\xhookrightarrow{}\bar{\mathcal{G}'}(K_{\nu}),$$ and $$\phi^{-1}\circ\pi:{\Gamma'}(\mathcal{O}_S')\to\bar{\Gamma}(\mathcal{O}_{S})\xhookrightarrow{}\bar{\mathcal{G}}(K_{\nu}),$$ for every $\nu\in S$ in order to obtain the following extensions
 $$\phi_{K_{\nu}}:\mathcal{G}(K_{\nu})\to\bar{\mathcal{G}'}(K_{\nu}),$$ and $$\phi^{-1}_{K_{\nu}}:{\mathcal{G}'}(K_{\nu})\to\bar{\mathcal{G}}(K_{\nu}),$$ with  (\ref{virextprop}) for some ideal $q\trianglelefteq\mathcal{O}_{S}=\mathcal{O}_{S'}$. 
 
 If the above extensions are surjective then they factor through the centre of their domains and one obtains 
 $$\bar{\phi}_{K_{\nu}}:\bar{\mathcal{G}}(K_{\nu})\to\bar{\mathcal{G}'}(K_{\nu}),$$ and $$\bar{\phi}^{-1}_{K_{\nu}}:\bar{{\mathcal{G}'}}(K_{\nu})\to\bar{\mathcal{G}}(K_{\nu}),$$ with  (\ref{virextprop}) for some ideal $q\trianglelefteq\mathcal{O}_{S}=\mathcal{O}_{S'}$.
 
 To check surjectivity one notices that the subgroups ${\Gamma}(\mathcal{O}_S)$, $\mathcal{G}(\mathcal{O}_{s})$ and ${\Gamma'}(\mathcal{O}_{S'})$, $\mathcal{G}'(\mathcal{O}_{S'})$ are dense in $\mathcal{G}(K_{\nu})$ and $\mathcal{G}'(K_{\nu})$ respectively. Moreover, since the weak Zariski topology is coarser than the Kac-Peterson topology, ${\Gamma}(\mathcal{O}_S)$, $\mathcal{G}(\mathcal{O}_{s})$ and ${\Gamma'}(\mathcal{O}_{S'})$, $\mathcal{G}'(\mathcal{O}_{S'})$ are dense in $\mathcal{G}(K_{\nu})$ and $\mathcal{G}'(K_{\nu})$ respectively with the weak Zariski topology as well (see \cite[Lemma 7.14.]{zbMATH06308047}). Consequently, $\bar{\Gamma}(\mathcal{O}_S)$, $\bar{\mathcal{G}}(\mathcal{O}_{s})$ and $\bar{\Gamma'}(\mathcal{O}_{S'})$, $\bar{\mathcal{G}'}(\mathcal{O}_{S'})$ are dense in $\bar{\mathcal{G}}(K_{\nu})$ and $\bar{\mathcal{G}'}(K_{\nu})$ respectively with respect to both the Kac-Peterson and the weak Zariski topologies. Now continuity of the extensions together with their local commensurability, Theorem~\ref{thmlcombc} and (\ref{virextprop}) implies surjectivity as follows. By  Theorem~\ref{thmlcombc} and (\ref{virextprop}) one has
 $$[\phi^{-1}(\bar{\mathcal{G}'}_{\alpha,\beta}(\mathcal{O}_{S'})):\phi^{-1}(\bar{\mathcal{G}'}_{\alpha,\beta}(\mathcal{O}_{S'}))\cap\bar{\Gamma}_{\pi^{\mathcal{G}}}(q)]<\infty,$$
 and hence
 $$[\bar{\mathcal{G}'}_{\alpha,\beta}(\mathcal{O}_{S'}):\bar{\mathcal{G}'}_{\alpha,\beta}(\mathcal{O}_{S'})\cap\phi(\bar{\Gamma}_{\pi^{\mathcal{G}}}(q))]<\infty.$$ This implies that the image  $\phi(\bar{\Gamma}_{\Pi^{\mathcal{G}}}(q))$ is dense in $\bar{\mathcal{G}'}(K_{\nu}).$ But by Theorem~\ref{thmsrbc} $\bar{\phi}_{K_{\nu}}$ and $\bar{\phi}^{-1}_{K_{\nu}}$ are continuous with respect to the weak Zariski topology hence surjectivity.
 
 By continuity of the extensions the isomorphism for the adjoint forms is obtained if one only shows that 
 $\bar{\phi}^{-1}_{K_{\nu}}\circ\bar{\phi}_{K_{\nu}}=\text{id}$ and $\bar{\phi}_{K_{\nu}}\circ\bar{\phi}^{-1}_{K_{\nu}}=\text{id}.$
 Since $\mathcal{G}$ and $\mathcal{G}'$ are simply connected, their fundamental rank 2  subgroups generate the whole group for any field. The same holds for $\bar{\mathcal{G}}$ and $\bar{\mathcal{G}}'$. Moreover, both of the Kac-Moody functors are irreducible. Hence, by symmetry of the argument it suffices to show that for a non-orthogonal pair of simple roots $\alpha,\beta\in\Pi^{\mathcal{G}},$ one has
 \begin{equation}\label{thmsteqli}
 (\bar{\phi}^{-1}_{K_{\nu}}\circ\bar{\phi}_{K_{\nu}})|_{\bar{\mathcal{G}}_{\alpha,\beta}(K_{\nu})}=\text{id}|_{\bar{\mathcal{G}}_{\alpha,\beta}(K_{\nu})}.   
 \end{equation}
 Note that if ${\phi}^{\alpha,\beta}_{K_{\nu}}$ and $\bar{\phi}^{\alpha,\beta}_{K_{\nu}}$ denote ${\phi}_{K_{\nu}}|_{\mathcal{G}_{\alpha,\beta}(K_{\nu})}$ and $\bar{\phi}_{K}|_{\mathcal{G}_{\alpha,\beta}(K_{\nu})}$ respectively, then by (\ref{virextprop}) and uniqueness of local extensions the following diagram is commutative:

\begin{equation*}
\begin{tikzpicture}
  \matrix (m) [matrix of math nodes,row sep=3em,column sep=4em,minimum width=2em]
  {
     \mathcal{G}_{\alpha,\beta}(K_{\nu}) & B^{\alpha,\beta}_{K_{\nu}} \\
    \bar{\mathcal{G}}_{\alpha,\beta}(K_{\nu}) &  \\};
  \path[-stealth]
    (m-1-1) edge node [left] {$\pi$} (m-2-1)
            edge node [above] {${\phi}^{\alpha,\beta}_{K_{\nu}}$} (m-1-2)
    
    (m-2-1) edge node [below] {${\bar{\phi}^{\alpha,\beta}_{K_{\nu}}}$} (m-1-2)
           ;
\end{tikzpicture} 
 \end{equation*}
 where $B^{\alpha,\beta}_{K_{\nu}}$ is obtained as in Theorem~\ref{thmlsrbc}. By Remark~\ref{remboundsemi} $B^{\alpha,\beta}_{K_{\nu}}$ can be modified such that the above morphisms are surjective and $B^{\alpha,\beta}_{K_{\nu}}$ lies inside the derived subgroup of a spherical Levi subgroup for any $\nu\in S=S'.$ 
Moreover, in the above diagram $\pi$ is a finite covering. Therefore, $$\text{Ker}(\bar{\phi}^{\alpha,\beta}_{K_{\nu}})\xhookrightarrow{i}\bar{\mathcal{G}}_{\alpha,\beta}(K_{\nu})\xrightarrow{\bar{\phi}^{\alpha,\beta}_{K_{\nu}}}B^{\alpha,\beta}_{K_{\nu}},$$
is a central extension. Recall that one may identify $B^{\alpha,\beta}_{K_{\nu}}$ with the image of $\bar{\phi}^{\alpha,\beta}_{K_{\nu}}$ to ensure surjectivity. {Now, since the elements in the kernel of the local extensions are central in rank 2  subgroups and the extensions coincide on overlaps, an element of the kernel of $\bar{\phi}^{\alpha,\beta}_{K_{\nu}}$ not only commutes with rank $1$ root subgroups of $\bar{\mathcal{G}}_{\alpha,\beta}(K_{\nu})$ but also with all root subgroups whose corresponding roots are non-orthogonal to $\pm\alpha$ and $\pm\beta$ and since the Kac-Moody functor is centred, the element in the kernel has to lie inside the centre of the ambient Kac-Moody group, namely the centre of $\bar{\mathcal{G}}(K_{\nu})$, which is trivial. Hence $\bar{\phi}^{\alpha,\beta}_{K_{\nu}}$ is an isomorphism.} Therefore, we have obtained (\ref{thmsteqli}). {Moreover, this also implies that 
$$\phi^{\alpha,\beta}_{K_{\nu}}:\mathcal{G}_{\alpha,\beta}(K_{\nu})\to B^{\alpha,\beta}_{K_{\nu}},$$
is a central isogeny.} 

To this end, one needs to consider the centres and how the extensions $\phi_{K_{\nu}}$ $\phi^{-1}_{K_{\nu}}$ behave on them. Let $n,m$ be the size and the rank of the generalized Cartan matrix corresponding to $\mathcal{G}$ respectively. Then the centre of $\mathcal{G}(K_{\nu})$ is the direct product of $n-m$ copies of the unit group $K_{\nu}^{\times}$. When the ideal $0\not= q\trianglelefteq\mathcal{O}_{S}$ from  (\ref{virextprop}) is proper then the intersection of  $\mathcal{G}(\mathcal{O}_{S})$ with the centre is finite hence negligible. When $q$ happens to be the whole ring $\mathcal{O}_{S}$ then by the Dirichlet unit theorem the intersection of $\mathcal{G}(\mathcal{O}_{S})$ with the centre is dense in the centre. Now since the extensions map centres to centres, they can be lifted from the adjoint forms to the simply connected forms such that they are compatible with $\phi$.
\end{proof}
The next theorem can be interpreted as a \textbf{commensurability} property in the Kac-Moody context.
\begin{theorem}[Commensurability]\label{thmcombc}
Let $\mathcal{G}$ be an irreducible simply connected $2$-spherical split Kac-Moody group functor of rank $\geq 2$. Let $\Phi$ be the set of real roots associated to $\mathcal{G}$ and let $\Pi$ and $\Pi'$ be two different systems of simple roots in $\Phi$. Then for every ideal $0\not= q\trianglelefteq\mathcal{O}_{S}$ there exists a non-trivial ideal $\bar{q}\trianglelefteq\mathcal{O}_{S}$ such that $$\Gamma_{\Pi'}(\bar{q})\leq\Gamma_{\Pi}(q).$$
\end{theorem}
\begin{proof}
For every non-orthogonal pair of simple roots $\alpha,\beta$ in $\Pi'$ by the local commensurability (Theorem~\ref{thmlcombc}) $\mathcal{G}_{\alpha,\beta}(\mathcal{O}_{S})\cap\Gamma_{\Pi}(q)$ has finite index in $\mathcal{G}_{\alpha,\beta}(\mathcal{O}_{S})$. Now the classical commensurability for $S$-arithmetic Chevalley groups implies that there exists an ideal $q_{\alpha,\beta}\trianglelefteq\mathcal{O}_{S}$ such that $\Gamma_{\alpha,\beta}(q_{\alpha,\beta})\leq\mathcal{G}_{\alpha,\beta}(\mathcal{O}_{S})\cap\Gamma_{\Pi}(q).$ Therefore, the ideal $\bar{q}:=\prod q_{\alpha,\beta}$ satisfies the desired property where $\alpha,\beta$ run through all pairs of non-orthogonal simple roots in $\Pi'$.
\end{proof}
Next result was announced in Remark~\ref{remrapits}(1). 
\begin{corollary}\label{corsupringindsys}
Let $\mathcal{G}$ be a simply connected irreducible $2$-spherical split Kac-Moody functor of rank $\geq 2$. Let $\mathcal{H}$ be a centred $2$-spherical split Kac-Moody functor. Let $\Phi^{\mathcal{G}_{re}}$ be the set of real roots of $\mathcal{G}$. Let $$\phi:\Gamma(\mathcal{O}_{S})\to\mathcal{H}(l)$$ be a group homomorphism where $l$ is a local field of characteristic  0. If there exists an admissible pair of simple roots for a system of simple roots $\Pi$ of $\Phi$ with respect to $\phi$, then there exists a unique extension of $\phi$ to $\mathcal{G}(l)$ with (\ref{virextprop}) for any system of simple roots in $\Phi$. 
\end{corollary}
\begin{proof}
Since one can apply the same arguments to $\Gamma_{\Pi}(q)$ for any ideal $0\not= q\trianglelefteq\mathcal{O}_{S}$ to obtain the unique extension as in Theorem~\ref{thmsrbc}, the corollary follows from Theorem~\ref{thmcombc}.
\end{proof}

\appendix 
\section{Continuous morphisms of local fields}\label{ApA}
The following result is well-known. We reproduce its proof for the convenience of the reader. 
\begin{lemma}\label{aplemtiso}
Let $K$ be an algebraic number field and let $\nu$ be a valuation. Let  $\phi:({K}_{\nu}, +)\to({K}_{\nu}, +)$ be a non-trivial continuous group homomorphism with respect to the standard topology on the completion ${K}_{\nu}$. Then $\phi$ is open.
\end{lemma}
\begin{proof} 
\textbf{(non-archimedean case):} Let $\nu$ be a non-archimedean valuation.
Let $\phi(1_{\mathbb{Q}_{p}})=x$ for an embedding of $\mathbb{Q}_{p}$ in $K_{\nu}$ where $p$ in the characteristic of the residue field of $K_{\nu}$ and $x\in K_{\nu}.$ Since $\phi$ is a group homomorphism, for any $n\in\mathbb{Z}\subset\mathbb{Q}_{p},$ one has $\phi(n)=n\cdot x.$ Since $\mathbb{Z}$ is dense in $\mathbb{Z}_{p}\subset\mathbb{Q}_{p},$ and $\phi$ is continuous, the same holds for any $z\in\mathbb{Z}_{p},$ namely $\phi(z)=z\cdot x.$ But for any $q\in\mathbb{Q}_{p},$ there exists an integer $m\in\mathbb{Z}$ such that $m\cdot q\in\mathbb{Z}_p,$ hence $\phi$ is a multiplication restricted to any copies of $\mathbb{Q}_{p}$ in $K_{\nu}$. Now since $K_{\nu}$ is an $n=[K_{\nu}:\mathbb{Q}_{p}]$ ($n<\infty$) dimensional  extension of $\mathbb{Q}_p$, one can readily show that for a $\mathbb{Q}_{p}$-basis $\{e_1,\cdots, e_n\}$ and any  $X=\alpha_{1}e_1+\cdots+\alpha_{n}e_n\in K_{\nu}$ one has 
$$\phi(X)=\alpha_{1}\phi(e_1)+\cdots+\alpha_{n}\phi(e_n).$$ The lemma follows from the fact that $K_{\nu}$ is isomorphic to $n$ direct product copies of $\mathbb{Q}_{p}$ as topological vector spaces.

\textbf{(archimedean case):} Let $\nu$ be an archimedean valuation. It suffices to consider only $K_{\nu}=\mathbb{R}$. When $K_{\nu}=\mathbb{C},$ the same arguments as applied to $\mathbb{Q}_{p}$ and $K_{\nu}$ can be applied to $\mathbb{R}$ and $\mathbb{C}.$ Now let $\phi(1)=x$ for some $x\in\mathbb{R}.$ Then for every $n\in\mathbb{Z},$ $\phi(n)=n\cdot x$ and hence for any rational number $q=n/m\in\mathbb{Q}$ ($n$ and $m$ coprime) one has $m\phi(q)=\phi(m\cdot q)=\phi(n)=n\cdot x$ and hence $\phi(q)=q\cdot x$ for any $q\in\mathbb{Q}.$ Since $\phi$ is continuous and $\mathbb{Q}$ is dense in $\mathbb{R},$ one has $\phi(r)=r\cdot x$ for any $r\in\mathbb{R}.$
\end{proof}


\begin{thebibliography}{10}
	
	\bibitem{zbMATH05288866}
	P.~{Abramenko} and K.~S. {Brown}:
	\newblock {\em {Buildings. Theory and applications}}, {\bf 248}
	\newblock Berlin: Springer (2008).
	
	\bibitem{zbMATH05541531}
	P.-E. {Caprace}:
	\newblock {``Abstract'' homomorphisms of split Kac-Moody groups},
	\newblock {\em {Mem. Am. Math. Soc.}} \textbf{924}:84 (2009).
	
	\bibitem{zbMATH05072607}
	P.-E. {Caprace} and B.~{M\"uhlherr}:
	\newblock {Isomorphisms of Kac-Moody groups which preserve bounded subgroups},
	\newblock {\em {Adv. Math.}} \textbf{206}(1) (2006) 250--278.
	
	\bibitem{zbMATH03111938}
	C.~{Chevalley}:
	\newblock {Sur certains groupes simples},
	\newblock {\em {Tohoku Math. J. (2)}} \textbf{7} (1955) 14--66, .
	
	\bibitem{zbMATH05218470}
	M.~W. {Davis}:
	\newblock {\em {The geometry and topology of Coxeter groups}},
	\newblock Princeton, NJ: Princeton University Press (2008).
	
	\bibitem{parsa2015strong}
	A.~Farahmand~Parsa:
	\newblock {\em Strong rigidity of arithmetic {Kac-Moody} groups},
	\newblock Justus-Liebig-Universit{\"a}t (2015).
	
	\bibitem{zbMATH06667742}
	A.~{Farahmand Parsa}, M.~{Horn}, and R.~{K\"ohl}:
	\newblock {Isomorphisms and rigidity of arithmetic Kac-Moody groups},
	\newblock {\em {J. Lie Theory}} \textbf{26}(4) (2016) 1079--1105.
	
	\bibitem{zbMATH06308047}
	T.~{Hartnick}, R.~{K\"ohl}, and A.~{Mars}:
	\newblock {On topological twin buildings and topological split Kac-Moody
		groups},
	\newblock {\em {Innov. Incidence Geom.}} \textbf{13} (2013) 1--71.
	
	\bibitem{zbMATH03572914}
	G.~A. {Margulis}:
	\newblock {Discrete groups of motions of manifolds of nonpositive curvature},
	\newblock {\em {Transl., Ser. 2, Am. Math. Soc.}} \textbf{109} (1977) 33--45.
	
	\bibitem{zbMATH00049189}
	G.~A. {Margulis}:
	\newblock {\em {Discrete subgroups of semisimple Lie groups}}, \textbf{17}
	\newblock Berlin etc.: Springer-Verlag (1991).
	
	\bibitem{zbMATH03541031}
	M.~S. {Raghunathan}:
	\newblock {On the congruence subgroup problem},
	\newblock {\em {Publ. Math., Inst. Hautes \'Etud. Sci.}} \textbf{46} (1976) 107--161.
	
	\bibitem{zbMATH04196274}
	A.~{Rapinchuk}:
	\newblock {The congruence subgroup problem for algebraic groups}, in:
	\newblock {Topics in algebra. Pt. 2: Commutative rings and algebraic groups,
		Pap. 31st Semester Class. Algebraic Struct., Warsaw/Poland 1988}, Banach Cent. Publ. \textbf{26} Part 2 (1990) 399-410.
	
	\bibitem{zbMATH01762600}
	B.~{R\'emy}:
	\newblock {\em {Groupes de Kac-Moody d\'eploy\'es et presque d\'eploy\'es}},
	\textbf{277}
	\newblock Paris: Ast\'erisque (2002).
	
	\bibitem{zbMATH06680202}
	R.~{Steinberg}:
	\newblock {\em{Lectures on Chevalley groups}}, \textbf{66}
	\newblock Providence, RI: American Mathematical Society (AMS) (2016).
	
	\bibitem{zbMATH01950200}
	B.~{Sury}:
	\newblock {\em {The congruence subgroup problem. An elementary approach aimed
			at applications}}, \textbf{24}
	\newblock New Delhi: Hindustan Book Agency (2003).
	
	\bibitem{zbMATH04142315}
	O.~I. {Tavgen'}:
	\newblock {Bounded generation of Chevalley groups over rings of S-integer
		algebraic numbers},
	\newblock {\em {Izv. Akad. Nauk SSSR, Ser. Mat.}} \textbf{54}(1) (1990) 97--122.
	
	\bibitem{zbMATH04017219}
	J.~{Tits}:
	\newblock {Uniqueness and presentation of Kac-Moody groups over fields},
	\newblock {\em {J. Algebra}} \textbf{105} (1987) 542--573.
	
	\bibitem{zbMATH03305793}
	B.~L. {van der Waerden}:
	\newblock {Algebra. 2. Teil.}
	\newblock {Unter Benutzung von Vorlesungen von E. Artin und E. Noether. 5.
		Aufl. der Modernen Algebra Heidelberger Taschenb\"ucher 23.
		Berlin-Heidelberg-New York: Springer- Verlag. XII, 300 S.} (1967).
	
	\bibitem{zbMATH00756383}
	A.~{Weil}:
	\newblock {\em {Basic number theory. Reprint of the 2nd ed. 1973}},
	\newblock Berlin: Springer-Verlag, reprint of the 2nd ed. 1973 edition (1995).
	
	\bibitem{zbMATH03911340}
	R.~J. {Zimmer}:
	\newblock {\em {Ergodic theory and semisimple groups}}, \textbf{81}
	\newblock Birkh\"auser/Springer, Basel (1984).
	
\end{thebibliography}
\end{document}